\documentclass[12pt]{amsart}
\usepackage{amsmath,amssymb,amsthm}
\usepackage{mathtools}
\usepackage{stmaryrd} 
\usepackage[inline]{enumitem}
\usepackage{graphicx,setspace}
\usepackage{tikz}
\usetikzlibrary{arrows}
\usepackage{wasysym}
\usepackage[makeroom]{cancel}
\usepackage{fullpage}
\usepackage{comment}
\usepackage{enumitem}
\usepackage{caption}
\allowdisplaybreaks
\newcommand{\bb}{\mathbb}

\theoremstyle{plain}
\newtheorem{theorem}{Theorem}[section]
\newtheorem{proposition}[theorem]{Proposition}

\newtheorem{lemma}[theorem]{Lemma}
\newtheorem{corollary}[theorem]{Corollary}

\newtheorem{counterexample}[theorem]{Counterexample}
\newtheorem*{claim*}{Claim}
\theoremstyle{definition}
\newtheorem{remark}[theorem]{Remark}
\newtheorem{definition}[theorem]{Definition}

\newcommand\abs[1]{\left|#1\right|}

\newcommand{\floor}[1]{\left\lfloor#1\right\rfloor}

\def\F{\mathbb{F}}

\def\Z{\mathbb{Z}}
\def\Q{\mathbb{Q}}

\def\S{\mathbb{S}}

\def\ord{\operatorname{ord}}

\DeclareMathOperator{\Spec}{Spec}

\begin{document}
\title{Anomalous Primes and the Elliptic Korselt Criterion}

\author[L. Babinkostova]{L. Babinkostova}
\address{Department of Mathematics\\ Boise State University\\
Boise, ID 83725, USA}
\email{liljanababinkostova@boisestate.edu}

\author[J. C.\ Bahr]{J. C.\ Bahr}
\address{Department of Mathematics\\ University of California, Los Angeles\\
Los Angeles, California 90095}
\email{jbahr@ucla.edu }

\author[Y. H. Kim]{Y. H. Kim}
\address{Department of Mathematics\\ Columbia University\\
New York, NY 10027, USA}
\email{yujin.kim@columbia.edu }

\author[E. Neyman]{E. Neyman}
\address{Department of Mathematics\\ Princeton University\\
Princeton, NJ 08544, USA}
\email{eneyman@princeton.edu}

\author[G. K.\ Taylor]{G. K.\ Taylor}
\address{Department of Mathematics, Statistics, and Computer Science\\ University of Illinois, Chicago\\
Chicago, IL 60607, USA}
\email{gtaylo9@uic.edu}

\thanks{Supported by National Science Foundation under the Grant number DMS-1062857}
\thanks{$^{\S}$ Corresponding Author: liljanababinkostova@boisestate.edu}

\begin{abstract}
We explore the relationship between elliptic Korselt numbers of Type I, a class of pseudoprimes introduced by Silverman in \cite{Korselt}, and anomalous primes. We generalize a result in \cite{Korselt} that gives sufficient conditions for an elliptic Korselt number of Type I to be a product of anomalous primes. Finally, we prove that almost all elliptic Korselt numbers of Type I of the form $n=pq$ are a product of anomalous primes.
\end{abstract}

\maketitle

\section{Introduction}\label{intro}

In 1989, Gordon \cite {DG,DG2} defined elliptic pseudoprimes for CM elliptic curves, the first to do so according to Silverman \cite{Korselt}. For a full discussion of Gordon's approach, see \cite[Remark 4]{Korselt} and the works cited there. In this paper, we work with the definition introduced by Silverman \cite{Korselt} which is well-defined for arbitrary elliptic curves.

Namely, for a given elliptic curve $E/\Q$ and point $P\in E(\Z/n\Z)$ a natural number $n$ is an \emph{elliptic pseudoprime with respect to} $P \in E$ if $n$ has at least two distinct prime factors, $E$ has good reduction at every prime $p$ dividing $n$, and $(n+1-a_n)P \equiv 0\pmod n$ where $a_n$ denotes the $n^{\mathrm{th}}$ coefficient of the \emph{L}-series of $E/\Q$. When we write $E(\Z / n\Z)$, we follow the convention of \cite[Remark 2]{Korselt} by assuming that $E$ has good reduction at all primes $p$ dividing $n$. In this case, a minimal Weierstrass equation defines a smooth group scheme $E  \to \Spec(\Z / n\Z)$ (see \cite[Section IV.5]{silBook2} for details of the construction). Then the $\Z / n\Z$-points of this group scheme form an abelian group, which we denote by $E(\Z / n\Z)$.

A composite integer $n$ is an \emph{elliptic Carmichael number} for a given elliptic curve $E/\Q$ if $n$ is an elliptic pseudoprime for every point $P \in E(\Z / n\Z)$ (note that by our convention, $E$ must have good reduction at the primes dividing $n$). In analogy with classical case of the Korselt criterion for Carmichael numbers,  Silverman \cite{Korselt} gives two Korselt-type criteria for elliptic Carmichael numbers, introducing the notion of an elliptic Korselt number of Type I and of Type II. The elliptic Korselt criterion of Type I is a practical sufficient condition for elliptic Carmichael numbers, given the prime factorization of the number.

For a given elliptic curve $E / \Q$, a prime $p$ is \emph{anomalous} if $E$ has good reduction at $p$ and $\#E(\F_p) = p$. In \cite[Proposition 17]{Korselt}, Silverman proves that if $n = pq$ is a Type I elliptic Korselt number for $E$ and $p$ is not too small with respect to $q$, then $p$ and $q$ are anomalous primes for $E$. Silverman notes that Type I elliptic Korselt numbers of the form $pq$ are interesting since there are no classical Carmichael numbers of this form. In this paper, we further explore the connection between squarefree elliptic Korselt numbers and anomalous primes.

In Section 2, we prove a generalization of \cite[Proposition 17]{Korselt}. We show that if $n = p_1\cdots p_m$ is a squarefree Type I elliptic Korselt number with $p_1 < \cdots < p_m$ and $\frac{\sqrt{p_m}}{4^m} \leq p_1\cdots p_{m-1} \leq 4^m$, then $p_m$ is anomalous and $a_n = 1$. Hence all but an even number of the primes $p_i$ are anomalous, and if $p_i$ is not anomalous, then $a_{p_i} = -1$. Furthermore, we note an error in the proof of \cite[Proposition 17]{Korselt}, providing a counterexample and the corrected statement. In particular, we show that if $n = pq$ is an elliptic Korselt number of Type I with $p < q$ for $E/\Q$ and $13 \leq p \leq \sqrt{q}/16$, then $p$ and $q$ are anomalous for $E$. 

In Section 3, we prove that for an elliptic Korselt number of Type I for $E/\Q$ of the form $n=pq$ where $p<q$ are prime, the probability that $p$ and $q$ are anomalous approaches 1 as $p,q \to \infty$. Our result relies on a result proven in the preprint \cite[Section 6]{BH-EK} which appeared as a conjecture in an early draft of this paper. Computational evidence for our original conjecture is collected in an appendix.

Finally, we combine our results with a result from \cite{Q} to show that (assuming the Hardy-Littlewood Conjecture) there are infinitely many Type I elliptic Korselt numbers for any curve $E: y^2 = x^3 + D$, where $D\in \mathbb{Z}$ is neither a square nor a cube in $\mathbb{Q}(\sqrt{-3})$ and $D \neq 80d^6$ for any $d \in \mathbb{Z}[(1+\sqrt{-3})/2]$.

\section{Squarefree Elliptic Korselt Numbers of Type I}\label{type 1}

The classical notions of pseudoprimes and Carmichael numbers are related to the orders of numbers in the multiplicative group $(\Z/n\Z)^*$. These concepts can be generalized to other algebraic groups, such as elliptic curves. The notions of elliptic pseudoprimes and elliptic Carmichael numbers were introduced in \cite{DG2} for curves with complex multiplication.

In \cite{Korselt} these notions were extended to arbitrary elliptic curves $E/\Q$. The definition of an elliptic pseudoprime for an arbitrary elliptic curve is as follows: let $n$ be a positive integer greater than $1$, let $E/\Q$ be an elliptic curve given by a minimal Weierstrass equation, and let $P\in E(\Z/n\Z)$. Write the $L$-series of $E/\Q$ as $L(E/\Q,s) =\sum a_n/n^s$. Then $n$ is an \emph{elliptic pseudoprime} for $(E,P)$ if $n$ has at least two distinct prime factors, $E$ has good reduction at every prime dividing $n$, and $P$ is $(n+1-a_n)$-torsion in $E(\Z/n\Z)$. Following the analogy with classical pseudoprimes, if $n=p$ is a prime of good reduction, the last condition holds trivially. Indeed, $E(\Z/p\Z)$ is an elliptic curve over $\F_p$ of order $p+1-a_p$, so every point is $(p+1-a_p)$-torsion.

The definition of an elliptic Carmichael number for an arbitrary elliptic curve is as follows: let $n$ be a positive integer greater than $1$ and let $E/\Q$ be an elliptic curve. Then $n$ is an \emph{elliptic Carmichael number} for $E$ if $n$ is an elliptic pseudoprime for $(E,P)$ for every point $P\in E(\Z/n\Z)$. In this section, we only consider integers $n$ that are coprime with $2$ and $3$. % Added the sentence beginning with ``indeed" -Eric.
In classical number theory, Korselt's criterion---which is satisfied by a composite number $n$ if $(p - 1) \mid (n - 1)$ for every prime $p$ dividing $n$---can be used to test for Carmichael numbers. In \cite{Korselt}, Silverman introduces an analogous criterion for elliptic curves. 

\begin{definition} \label{korseltDef}
A positive integer $n$ is called  an \emph{elliptic Korselt number of Type I} if it has at least two distinct prime factors, such that for every prime dividing $n$, the following hold:
\begin{enumerate}[label=(\arabic*)]
\item $E$ has good reduction at $p$
\item $p + 1 - a_p \mid n + 1 - a_n$
\item $\ord_p(a_n - 1) \ge \ord_p(n) - \begin{cases}1 & a_p \not \equiv 1 \pmod{p} \\ 0 & a_p \equiv 1 \pmod{p}\end{cases}$.
\end{enumerate}
\end{definition}

Here, $a_p$ is the Frobenius trace of $E(\F_p)$ as usual, and $a_n$ is the $n^{th}$ coefficient of the $L$-series of $E/\Q$; for how to compute this coefficient, see \cite{Z}. In particular, $a_n$ is a multiplicative function when $n$ is square-free, in the sense that if $n = \prod_i p_i$ for distinct $p_i$, then $a_n = \prod_i a_{p_i}$. Finally, $\ord_p(n)$ denotes the highest power of $p$ that appears in the prime factorization of $n$, with $\ord_p(0) = \infty$. In \cite{Korselt} it has been shown that any number satisfying this elliptic Korselt criterion is an elliptic Carmichael number, but the converse need not be true. 

\begin{proposition}
If $n$ is an elliptic Korselt number of Type I for an elliptic curve $E$, then $n$ is an elliptic Carmichael number for $E$.
\begin{proof}
See \cite[Proposition 11]{Korselt}.
\end{proof}
\end{proposition}

Recall that an anomalous prime for an elliptic curve $E / \mathbb{Q}$ is a prime such that $E$ has good reduction at $p$ and $\#E(\bb{F}_p) = p$ (or equivalently, $a_p = 1$).

\begin{proposition} \label{anomProdKorselt}
Let $E$ be an elliptic curve and let $p_1, p_2, \dots, p_m$ be distinct anomalous primes for $E$. Then $n = \prod_{i = 1}^m p_i$ is an elliptic Korselt number of Type I for $E$.
\end{proposition}
\begin{proof}
The first condition of Definition \ref{korseltDef} is satisfied since elliptic curves have good reduction at anomalous primes. The second condition is satisfied since $a_n = \prod_{i = 1}^m a_{p_i} = 1$, and each $p_i$ divides $n$. The third condition is satisfied because for each $i$, $\ord_{p_i}(a_n - 1) = \ord_{p_i}(0) = \infty$.
\end{proof}

The converse of Proposition \ref{anomProdKorselt} is not true: not all elliptic Korselt numbers of Type I for an elliptic curve $E$ are products of distinct primes which are anomalous for $E$. However, for a product of two distinct primes $n = pq$  there are conditions on $p$ and $q$ under which both $p$ and $q$ must be anomalous. A result of this form was obtained in \cite[Proposition 17]{Korselt}. This proposition states if $n = pq$ is Type I elliptic Korselt for $E$ with $17 < p < \sqrt{q}$, then $a_p=a_q=1$ i.e.  $p$ and $q$ are anomalous for $E$. However, there is a mistake in the proof, resulting in incorrect bounds. A counterexample is included below.

\begin{counterexample} Let $E: y^2 = x^3 + 1$, $p = 53$ and $q = 2971$. We have $a_p = 0$ and $a_q = 56$, and $pq$ an elliptic Korselt number of Type I for $E$. However, $17 < p < \sqrt{q}$ is satisfied. 
\end{counterexample} 

The remainder of this section is devoted to proving a generalization of \cite[Proposition 17]{Korselt} to squarefree Type I elliptic Korselt numbers. In particular, we include conditions on distinct primes $p,q$ so that if $pq$ is Type I elliptic Korselt for a curve $E$, then $p$ and $q$ are anomalous for $E$.

\begin{theorem} \label{gen_silv}
Let $E/\Q$ be an elliptic curve and let $n = p_1 p_2 \dots p_m$ be an elliptic Korselt number of Type I for $E$ such that $5 \le p_1 < p_2 < \dots < p_m$, where $m \ge 2$. Then one of the following conditions is satisfied:
\begin{enumerate}
    \item $p_1\cdot \dots \cdot p_{m-1} \leq 4^m$
    \item $a_{p_m} = 1$, and for $1 \le i \le m - 1$, $a_{p_i} = -1$ for an even number of values of $i$ and the remaining traces are equal to $1$
    \item $p_1\cdot \dots \cdot p_{m-1} \geq \frac{\sqrt{p_m}}{4^m}$.
\end{enumerate}
\end{theorem}
\begin{proof}

Assume $p_1\cdots p_{m-1} = \frac{n}{p_m} > 4^m$. We show that one of the two remaining conditions of the theorem are satisfied. We have
\begin{align*}
    n+1-a_n &= p_1\cdots p_{m} + 1 - a_{p_1}\cdots a_{p_m} \\ 
    &= \frac{n}{p_m}(p_m+1-a_{p_m})-\frac{n}{p_m}+a_{p_m}\frac{n}{p_m}+1-a_n.
\end{align*}
By the divisibility criterion of Type I elliptic Korselt numbers, we have $p_m + 1 - a_{p_m} \mid n + 1 - a_n$, so
\begin{align*}
    \left(p_m+1-a_{p_m}\right) \mid \left(-\frac{n}{p_m}+a_{p_m}\frac{n}{p_m}+1-a_n\right).
\end{align*}
We now consider two cases: $-\frac{n}{p_m}+a_{p_m}\frac{n}{p_m}+1-a_n = 0$ and $-\frac{n}{p_m}+a_{p_m}\frac{n}{p_m}+1-a_n \neq 0$.
\begin{description}
\item[Case 1] $-\frac{n}{p_m}+a_{p_m}\frac{n}{p_m}+1-a_n = 0$.
\end{description}
In this case, we have
\begin{equation} \label{a_n-1}
\frac{n}{p_m}(a_{p_m} - 1) = a_n - 1.
\end{equation}
Suppose for sake of contradiction that $a_{p_m} \neq 1$. We will show that this leads to $\frac{n}{p_m} \le 4^m$. We have
\begin{align*}
\frac{n}{p_m} = \frac{a_n - 1}{a_{p_m} - 1} &= \frac{(a_{p_1}\cdots a_{p_{m - 1}}) a_{p_m} - 1}{a_{p_m} - 1}.
\end{align*}

For simplicity of notation, let $r$ denote $a_{p_1}\cdots a_{p_{m - 1}}$. Since $a_{p_m} \neq 1$ is an integer, the possible values of $\frac{n}{p_m}$ in terms of $r$ are
\[\dots, \frac{2r + 1}{3}, \frac{r + 1}{2}, 1, 2r - 1, \frac{3r - 1}{2}, \frac{4r - 1}{3}, \dots,\]
where $a_{p_m} \in \Z-\{1\}$, respectively. If $r < 0$ then the maximum of these values is $1$, so the desired inequality is clear. Assume instead that $r$ is positive. Then $\frac{n}{p_m}$ is maximized when $a_{p_m} = 2$, in which case $\frac{n}{p_m} = 2r - 1$. Now by Hasse's theorem, $r \le 2^{m - 1}\sqrt{\frac{n}{p_m}}$, and so
\[\frac{n}{p_m} \le 2 \cdot 2^{m - 1}\sqrt{\frac{n}{p_m}} - 1 < 2^m\sqrt{\frac{n}{p_m}},\]
so $\frac{n}{p_m} \le 4^m$, as desired. However, by assumption, $\frac{n}{p_m} > 4^m$. 
Thus, we have a contradiction: if $a_{p_m}-p_1\cdots p_{m-1} -a_{p_1}\cdots a_{p_m}+1 = 0$, then $a_{p_m}$ must be $1$. We can say more: if $a_{p_m} = 1$, then by (\ref{a_n-1}), $a_{p_1}\cdots a_{p_{m-1}} =1$. Thus, an even number of traces $a_{p_i}$ for $1 \leq i \leq m-1$ must be equal to $-1$, while the rest of these traces must be equal to 1.

\begin{description}
\item[Case 2] $-\frac{n}{p_m}+a_{p_m}\frac{n}{p_m}+1-a_n \neq 0$.
\end{description}
Since $p_m+1-a_{p_m} \mid -\frac{n}{p_m}+a_{p_m}\frac{n}{p_m}+1-a_n$, we have
\begin{align*}
\abs{p_m+1-a_{p_m}} &\leq \abs{\frac{n}{p_m}a_{p_m}-\frac{n}{p_m}-a_{p_1}\dots a_{p_m}+1} \\
p_m+1-2\sqrt{p_m} \leq \abs{p_m+1-a_{p_m}} &\leq  2\frac{n}{p_m}\sqrt{p_m}+\frac{n}{p_m}+2^m\sqrt{p_m}\sqrt{\frac{n}{p_m}}-1.
\end{align*}
Subtracting the left-most quantity from the right-most gives:
\begin{align*}
(2\sqrt{p_m}+1)\frac{n}{p_m} + 2^m\sqrt{p_m}\sqrt{\frac{n}{p_m}} - p_m-2+2\sqrt{p_m} \geq 0.
\end{align*}
Solving the quadratic equation for $\sqrt{\frac{n}{p_m}}$ yields:
\begin{equation} \label{untoned_quads}
\sqrt{\frac{n}{p_m}} \geq \frac{-2^m\sqrt{p_m}+\sqrt{4^{m}p_m - 4(2\sqrt{p_m} + 1)(-p_m-2+2\sqrt{p_m})}}{2(2\sqrt{p_m}+1)}.
\end{equation}
Using the following claim, we will show the right-hand side of (\ref{untoned_quads}) is at least $\frac{1}{2^m}p_m^{1/4}$.
\begin{claim*} \label{uglyEq}
\begin{equation} \label{uglyIneq}
\left( 8 - \frac{16}{4^m} \right)p_m^{3/2} - 8p_m^{5/4} - \left( 12 + \frac{16}{4^m} \right)p_m - 4p_m^{3/4} + \left( 8 - \frac{4}{4^m} \right)p_m^{1/2} + 8 \ge 0.
\end{equation}
\end{claim*}

For $m = 2$ it can be verified with a computer algebra system that (\ref{uglyIneq}) is true when $p_2 \ge 19$. By assumption, $p_1 > 16$, so this is always the case. For $m = 3$ it can be verified with a computer algebra system that (\ref{uglyIneq}) is true when $p_3 \ge 13$. Note that $p_3 > 11$ because otherwise we have $p_1p_2 \le 5 \cdot 7 = 35$, contradicting the initial assumption. Thus, the claim holds for $m = 3$.

Now, let $f(m, p_m)$ be the left-hand side of (\ref{uglyIneq}).
Observe that if $p_m > 0$ is held constant and $m$ is increased, then $f(m, p_m)$ increases. This is because we may write
\[f(m, p_m) = g(p_m) - \frac{16}{4^m}p_m^{3/2} - \frac{16}{4^m}p_m - \frac{4}{4^m}p_m^{1/2},\]
where
\[g(p_m) = 8p_m^{3/2} - 8p_m^{5/4} - 12p_m - 4p_m^{3/4} + 8p^{1/2} + 8.\]
and as $m$ increases, $-\frac{16}{4^m}$ and $-\frac{4}{4^m}$ increase. Thus, since $f(3, p_m) \ge 0$ for $p_m \ge 13$, $f(m, p_m) \ge 0$ for $p_m \ge 13$ for all $m > 3$. Since $13$ is the fourth prime greater than or equal to $5$, $p_m \ge 13$ for all $m > 3$. This completes the proof of the claim.

Thus, we have 
\[\left( 8 - \frac{16}{4^m} \right)p_m^{3/2} - 8p_m^{5/4} - \left( 12 + \frac{16}{4^m} \right)p_m - 4p_m^{3/4} + \left( 8 - \frac{4}{4^m} \right)p_m^{1/2} + 8 \ge 0,\]
which implies 
\begin{align*}
8p_m^{3/2} - 12p_m + 8p_m^{1/2} + 8 &\ge \frac{16}{4^m}p_m^{3/2} + \frac{16}{4^m}p_m + \frac{4}{4^m}p_m^{1/2} + 8p_m^{5/4} + 4p_m^{3/4}\\
4^mp_m + 8p_m\sqrt{p_m} - 12p_m + 8\sqrt{p_m} + 8 &\ge 4^mp_m + \frac{4}{4^m}(2\sqrt{p_m} + 1)^2\sqrt{p_m} + 4(2\sqrt{p_m} + 1)p_m^{3/4}\\ %changed 2 \cdot 2 to 4 in the last coefficient
4^mp_m + 8p_m\sqrt{p_m} - 12p_m + 8\sqrt{p_m} + 8 &\ge \left( \frac{2}{2^m}(2\sqrt{p_m} + 1)p_m^{1/4} + 2^m\sqrt{p_m} \right)^2.
\end{align*}
The right-hand side above is positive and smaller than the left-hand side, so the left-hand side is also positive. We take the square root of both sides.
\begin{align*}
\sqrt{4^mp_m + 8p_m\sqrt{p_m} - 12p_m + 8\sqrt{p_m} + 8} &\ge \frac{2}{2^m}(2\sqrt{p_m} + 1)p^{1/4} + 2^m\sqrt{p_m}\\
\sqrt{4^mp_m - 4(2\sqrt{p_m} + 1)(-p_m - 2 + 2\sqrt{p_m})} &\ge \frac{2}{2^m}(2\sqrt{p_m} + 1)p^{1/4} + 2^m\sqrt{p_m}\\
\frac{-2^m\sqrt{p_m} + \sqrt{4^mp_m - 4(2\sqrt{p_m} + 1)(-p_m - 2 + 2\sqrt{p_m})}}{2(2\sqrt{p_m} + 1)} &\ge \frac{1}{2^m}p_m^{1/4}.
\end{align*}
Thus, $\frac{n}{p_m} = p_1\cdots p_{m-1} \geq \frac{\sqrt{p_m}}{4^m}$, concluding the proof of Theorem~\ref{gen_silv}.
\end{proof}

Note that for $m \ge 4$ the inequality $p_1\cdots p_{m-1} \le 4^m$, i.e. the first condition of Theorem~\ref{gen_silv}, is never satisfied.

\begin{remark}
Theorem~\ref{gen_silv} can be restated as follows. Let $E$ be an elliptic curve and let $n = p_1 p_2 \cdots p_m$ be an elliptic Korselt number of Type I for $E$ such that $5 \le p_1 < p_2 < \dots < p_m$, for $m \ge 2$. If $4^m < p_1 \cdots p_{m - 1} < \frac{\sqrt{p_m}}{4^m}$, then $a_{p_m} = 1$ and for $1 \le i \le m - 1$, $a_{p_i} = -1$ for an even number of values of $i$ and $a_{p_i} = 1$ for the remaining values. 
\end{remark}

The following corollary of Theorem~\ref{gen_silv} corrects the claim of Proposition 17 in \cite{Korselt}.%, to which $n = 53 \cdot 2971$ for $E: y^2 = x^3 + 1$ serves as a counterexample.
\begin{corollary} \label{fix_silv}
Let $E$ be an elliptic curve and let $n = pq$ be an elliptic Korselt number of Type I for $E$ such that $p < q$. Then one of the following conditions holds:
    \begin{itemize}
        \item $p \leq 13$
        \item $p$ and $q$ are anomalous for $E$.
        \item $p \geq \frac{\sqrt{q}}{16}$
    \end{itemize}
\end{corollary}

\section{Elliptic Korselt Numbers of Type I of the Form $pq$}\label{KorseltAnom}

In Proposition~\ref{anomProdKorselt}, we showed that any product of distinct anomalous primes of an elliptic curve is an elliptic Korselt number of Type I. Corollary \ref{fix_silv} gives sufficient conditions for when an elliptic Korselt number of Type I of the form $n = pq$ is a product of anomalous primes. In this section, we show that the probability that an elliptic Korselt number of Type I of the form $n=pq$ is a product of anomalous primes goes to $1$ as $n \to \infty$.

Part of our proof relies on the following proposition which is proven in the preprint \cite[Corollary 6.18]{BH-EK}. This particular statement appeared as a conjecture in an early draft of this paper which included numerical evidence to support the conjecture. We have relegated that numerical evidence to an appendix.

\begin{proposition} \label{odc}
    For $N \ge 7$, let $5 \le p, q \le N$ be distinct primes chosen uniformly, and let $n = pq$. Let $E(\Z/n\Z)$ be an elliptic curve chosen uniformly from the set of elliptic curves defined over $\Z/n\Z$ with good reduction over $\F_p$ and $\F_q$ such that $\# E(\F_p)$ and $\# E(\F_q)$ divide $n + 1 - a_n$.
    Then
    \[\lim \limits_{N \to \infty} \Pr[\# E(\Z/n\Z) = n + 1 - a_n] = 1.\]
\end{proposition}

Note that $\# E(\Z/n\Z) = (p + 1 - a_p)(q + 1 - a_q)$ and $n + 1 - a_n = pq + 1 - a_p a_q$. We have the following heuristic justification for the conjecture. Note that by Hasse's bound, $p + 1 - a_p$ and $q + 1 - a_q$ are close in value to $p$ and $q$, respectively, and $pq + 1 - a_p a_q$ is close in value to $pq$. Thus, if $p + 1 - a_p$ and $q + 1 - a_q$ divide $n + 1 - a_n$ and their product is not equal to $n + 1 - a_n$, then $p + 1 - a_p$ and $q + 1 - a_q$ must share many factors; this should happen rarely.

From the conditions listed in Proposition~\ref{odc}, it is clear that $n$ satisfies the first two conditions of the elliptic Korselt of Type I criterion for $E$. In other words, $n$ is ``nearly" elliptic Korselt number of Type I. The lemma below states that when $p, q \geq 7$, the third elliptic Korselt condition is a redundancy given the first and second. We will need this and the following results to prove the Theorem \ref{prob_thm} of this section.

\begin{lemma} \label{case1}
    For $N \ge 7$, let $5 \le p, q \le N$ be uniformly chosen distinct primes, and let $n = pq$. Let $E(\Z/n\Z)$ be an elliptic curve chosen uniformly among those for which $n$ is an elliptic Korselt number of Type I. Then
    \[\lim \limits_{N \to \infty} \Pr[\text{$p$ is anomalous for $E$ and $q$ is not}] = 0.\]
\end{lemma}
\begin{proof}
Let $N$, $p$, $q$, and $E$ be as in the statement. Assume that $a_p = 1$ and $a_q \neq 1$. By the Korselt divisibility condition, we have that $p$ and $q + 1 - a_q$ divide $pq + 1 - a_q$. Since $p \mid pq + 1 - a_q$, we have $p \mid 1 - a_q$. Since $1 - a_q \neq 0$,
\[p \le \abs{1 - a_q} \le \abs{a_q} + 1 \le 2\sqrt{q} + 1 \le 2\sqrt{N} + 1.\]
The probability that a randomly chosen prime below $N$ is at most $2\sqrt{N} + 1$ goes to zero as $N\rightarrow \infty$. Since $p \le 2\sqrt{N} + 1$ is a necessary condition for $a_p = 1$ and $a_q \neq 1$ for E, it follows that the desired probability approaches zero.
\end{proof}

\begin{proposition} \label{korseltEquiv}
If $n = pq$ for distinct primes $p, q \ge 7$ and $E$ is an elliptic curve with good reduction over $\F_p$ and $\F_q$, then $n$ is an elliptic Korselt number of Type I for $E$ if and only if $p + 1 - a_p$ and $q + 1 - a_q$ divide $n + 1 - a_n$.
\end{proposition}
\begin{proof}
By the definition of Type I elliptic Korselt, the ``only if" direction holds. Suppose that $E$ has good reduction at $p$ and $q$ that $p + 1 - a_p$ and $q + 1 - a_q$ divide $n + 1 - a_n$. 
Since $n=pq$, $\ord_p(n) = 1$, so $a_p \not \equiv 1 \pmod{p}$ implies the third condition of the elliptic Korselt criterion is satisfied for $p$. 

Alternatively, if $a_p \equiv 1 \pmod{p}$, then by Hasse's theorem, $p$ being at least $7$ implies that $a_p = 1$. Thus, $p+1-a_p = p$ and $n + 1 - a_n = pq + 1 - a_q$, and so $p \mid 1 - a_q$. Thus,
\[\ord_p(a_n - 1) = \ord_p(a_q - 1) \ge 1 = \ord_p(n),\]
and so $p$ satisfies the third condition of the elliptic Korselt criterion. By analogy, $q$ satisfies the third condition as well, and so we are done.
\end{proof}

\begin{lemma} \label{case2}
For $N \ge 7$, let $5 \le p, q \le N$ be uniformly chosen distinct primes, and let $n = pq$. Let $E(\Z/n\Z)$ be an elliptic curve chosen uniformly among those for which $n$ is an elliptic Korselt number of Type I. Then
    \begin{align*}
        \lim \limits_{N \to \infty} \Pr[&p \text{ and } q \text{ are not anomalous for } E \text{ and } (p + 1 - a_p)(q + 1 - a_q) \neq n + 1 - a_n] = 0.
    \end{align*}
\end{lemma}
\begin{proof}
Let $N$, $p$, $q$, and $E$ be as in the lemma statement. By Proposition~\ref{korseltEquiv}, this is equivalent to the statement that $p$, $q$ and $E$ are selected in such a way that $E(\Z/nZ)$ has good reduction over $\F_p$ and $\F_q$, and $\# E(\F_p)$ and $\# E(\F_q)$ divide $n + 1 - a_n$.\footnote{If $p = 5$ or $q = 5$, this does not follow from Proposition~\ref{korseltEquiv}, but the probability of this happening approaches zero as $N\rightarrow \infty$, so we can ignore this case.} By Proposition~\ref{odc}, the probability that
\[\# E(\Z/nZ) = \#E(\F_p) \#E(\F_q) = (p + 1 - a_p)(q + 1 - a_q) \neq n + 1 - a_n\]
approaches zero as $N\rightarrow \infty$. Thus, the probability that this condition is satisfied \emph{and} $p$ and $q$ are not anomalous for $E$ also approaches zero, as desired.
\end{proof}

\begin{proposition} \label{mdphi}
    Let $n$ be a positive integer and let $S$ be a finite multiset of factors of $n$. For each $d \mid n$, let $m_d(S)$ be the number of multiples of $d$ in $S$. Then
    \[\sum \limits_{k \in S} k = \sum \limits_{d \mid n} m_d(S) \phi(d).\]
\end{proposition}
\begin{proof}
This is an induction on the number of elements of $S$. The theorem is clear for $\abs{S} = 0$; suppose it holds for $\abs{S} = r$. Now let $S$ have $r + 1$ elements and choose $k \in S$. Let $S'$ be $S$ with one fewer copy of $k$; the theorem holds for $S'$. Adding $k$ to $S'$ increments the left-hand sum by $k$ and the right-hand sum by $\sum_{d \mid k} \phi(d)$, since $m_d(S) = m_d(S') + 1$ for $d \mid k$ and $m_d(S) = m_d(S')$ for all other $d$. But $\sum_{d \mid k} \phi(d) = k$ \cite[Proposition 2.2.4]{IR}, so the statement holds for $S$. 
\end{proof}

\begin{lemma} \label{case3}
    For $N \ge 7$, let $5 \le p, q \le N$ be distinct primes chosen uniformly at random, and let $n = pq$. Let $E(\Z/n\Z)$ be an elliptic curve chosen uniformly among those for which $n$ is an elliptic Korselt number of Type I. Then
    \[\lim \limits_{N \to \infty} \Pr[\text{$p$ and $q$ are not anomalous for $E$ and } (p + 1 - a_p)(q + 1 - a_q) = n + 1 - a_n] = 0.\]
\end{lemma}
\begin{proof}
Let $N$, $p$, $q$, and $E$ be as in the lemma statement. We impose the additional restriction that $q \ge 67$; this does not affect our proof, since the probability of a randomly selected prime below $N$ being less than $67$ approaches $0$ as $N$ approaches $\infty$. Assume that $a_p \neq 1$, $a_q \neq 1$, and $(p + 1 - a_p)(q + 1 - a_q) = n + 1 - a_n$. We have
\begin{align*}
    (p + 1 - a_p)(q + 1 - a_q) &= pq + 1 - a_p a_q\\
    2a_p a_q + p + q &= pa_q + qa_p + a_p + a_q\\
    a_p &= \frac{p + q - (p + 1)a_q}{q + 1 - 2a_q}.
\end{align*}

Thus, $q + 1 - 2a_q$ divides $p + q - (p + 1)a_q$. Subtracting $q + 1 - 2a_q$ from the dividend, we have
\[q + 1 - 2a_q \mid p - pa_q + a_q - 1 = (p - 1)(1 - a_q).\]
Letting $x = a_q - 1$, $p' = p - 1$, and $q' = q - 1$, we find that $q' - 2x$ divides $p'x$. It follows that
\begin{equation} \label{divisEq}
\frac{q' - 2x}{\gcd(q' - 2x, x)} = \frac{q' - 2x}{\gcd(q', x)} \mid p'.
\end{equation}
We claim that the probability for randomly chosen $5 \le p, q \le N$ that there exists $x \in [-2\sqrt{q} - 1, 2\sqrt{q} - 1]$ such that the above property is satisfied approaches zero as $N\rightarrow \infty$; this is sufficient to prove our lemma.

To prove this claim, fix $q$ (and thus $q'$) and examine how many values of $p' < N$ (and thus $p \le N$) satisfy the condition in (\ref{divisEq}) for some $x$ in the interval.

For a fixed $x$, the number of values of $p'$ divisible by $\frac{q' - 2x}{\gcd(q', x)}$ is bounded above by
\[\frac{N}{\frac{q' - 2x}{\gcd(q', x)}} = \frac{N\gcd(q', x)}{q' - 2x} \le \frac{2N\gcd(q', x)}{q'}.\]
(The last step is justified by the fact that $q \ge 67$.) Thus, the total number number of values of $p'$ that are divisible by $\frac{q' - 2x}{\gcd(q', x)}$ for some $x \in [-2\sqrt{q} - 1, 2\sqrt{q} - 1]$ is at most
\[\sum \limits_{\substack{x \in [-2\sqrt{q} - 1, 2\sqrt{q} - 1] \\ x \neq 0}} \frac{2N\gcd(q', x)}{q'} \le 2 \sum \limits_{x = 1}^{\floor{2\sqrt{q} + 1}} \frac{2N\gcd(q', x)}{q'} = \frac{4N}{q'} \sum \limits_{x = 1}^{\floor{2\sqrt{q} + 1}} \gcd(q', x).\]
Now, let $g(k) = \sum_{x = 1}^k \gcd(x, k)$. We claim that
    \[\sum \limits_{x = 1}^{\floor{2\sqrt{q} + 1}} \gcd(q', x) \le g(q') \cdot \frac{2\sqrt{q} + 1}{q'}.\]

For $n$ implicit, define the multiset $S_{a, k} = \{\gcd(x, n) \mid x \in \{a, a + 1, \dots, a + k - 1\}\}$. Observe that for all $d \mid n$, holding $k$ constant, $m_d(S_{a, k})$ is minimal for $a = 1$. It follows from Proposition~\ref{mdphi} that
\begin{equation} \label{gcdak}
\sum \limits_{x = a}^{a + k - 1} \gcd(x, n)
\end{equation}
is minimized for $a = 1$. In particular, let $h(a)$ be (\ref{gcdak}) with $n = q'$ and $k = \floor{2\sqrt{q} + 1}$. Note that
\begin{equation} \label{sumofh}
h(1) + h(2) + \dots + h(q') = g(q') \cdot \floor{2\sqrt{q} + 1},
\end{equation}
since the fact that $\gcd(q', q' + x) = \gcd(q', x)$ means that for every $x \in \{1, 2, \dots, q'\}$, $x$ appears $\floor{2\sqrt{q} + 1}$ times in (\ref{sumofh}). Since $h(1)$ is the smallest value among the $q'$ values in (\ref{sumofh}), we obtain
\[\sum \limits_{x = 1}^{\floor{2\sqrt{q} + 1}} \gcd(q', x) \le g(q') \cdot \frac{2\sqrt{q} + 1}{q'},\]
as desired. Thus, we have
\[\sum \limits_{\substack{x \in [-2\sqrt{q} - 1, 2\sqrt{q} - 1] \\ x \neq 0}} \frac{2N\gcd(q', x)}{q'} \le \frac{4N}{q'} \sum \limits_{x = 1}^{\floor{2\sqrt{q} + 1}} \gcd(q', x) \le \frac{4N}{q'} g(q') \cdot \frac{2\sqrt{q} + 1}{q'}.\]

Now, the number of primes $p \le N$ is on the order of $\frac{N}{\log N}$. Thus, the probability that $p$ is chosen such that the above divisibility property is satisfied for some $x$ is
\[O\left( \frac{4 \log N}{q'} g(q') \cdot \frac{2\sqrt{q} + 1}{q'} \right) = O\left( \log N \cdot g(q') q^{\frac{-3}{2}} \right).\]

It is known that $g(k) = O(k^{1 + \epsilon})$ for every positive $\epsilon$ \cite[Theorem 3.2]{gcdsum}. Thus, the probability above is $O \left( \log N \cdot q^{\frac{-1}{2} + \epsilon} \right)$ for every positive $\epsilon$, as a function of $q$ and $N$.

Now we express the probability as a function of just $N$, randomly choosing $q$ to be a prime below $N$. The probability that $q \le N^\frac{1}{2}$ is on the order of
\[\frac{{N^{1/2}}/{\log N^{1/2}}}{{N}/{\log N}} = \frac{2 N^\frac{1}{2}}{N},\]
which is on the order of $N^\frac{-1}{2}$. If $q > N^\frac{1}{2}$, then the above probability is $O \left( \log N \cdot N^{\frac{-1}{4} + \epsilon} \right)$ for all $\epsilon$. Thus, the total probability is at most on the order of $N^\frac{-1}{2} + \log N \cdot N^{\frac{-1}{4} + \epsilon}$, which approaches zero as $N\rightarrow \infty$, and so we are done.
\end{proof}

\begin{theorem} \label{prob_thm}
    For $N \ge 7$, let $5 \le p, q \le N$ be uniformly chosen distinct primes, and let $n = pq$. Let $E(\Z/n\Z)$ be an elliptic curve chosen uniformly among those for which $n$ is an elliptic Korselt number of Type I. We have
    \[\lim \limits_{N \to \infty} \Pr[\text{$p$ and $q$ are anomalous primes for $E$}] = 1.\]
\end{theorem}

\begin{proof}
A result of Deuring \cite{D} states that for all primes $p$, for every integer $-2\sqrt{p} \le t \le 2\sqrt{p}$, there is an elliptic curve over $\F_p$ with order $p + 1 - t$. In particular, for every $p$ there is an elliptic curve that is anomalous over $\F_p$. Thus, for any two primes $p$ and $q$, we may use the Chinese Remainder Theorem to construct a curve over $\Q$ that is anomalous both when reduced over $\F_p$ and over $\F_q$. It follows by Proposition~\ref{anomProdKorselt} that for all $(p, q)$ there is a curve $E$ that makes $p$ and $q$ anomalous and therefore makes $n = pq$ elliptic Korselt number of Type I.

Suppose now that $n = pq$ is an elliptic Korselt number of Type I for some elliptic curve $E$. Then the cases in which $p$ and $q$ are not both anomalous primes for $E$ are as follows:
\begin{enumerate}[label=(\arabic*)]
    \item \label{case1stated} Exactly one of $p$ and $q$ is anomalous for $E$.
    \item \label{case2stated} Neither $p$ nor $q$ is anomalous for $E$, and $(p + 1 - a_p)(q + 1 - a_q) \neq n + 1 - a_n$.
    \item \label{case3stated} Neither $p$ nor $q$ is anomalous for $E$, and $(p + 1 - a_p)(q + 1 - a_q) = n + 1 - a_n$.
\end{enumerate}
Lemmas \ref{case1}, \ref{case2}, and \ref{case3} show that the probability that $p$, $q$, and $E$ satisfy cases \ref{case1stated}, \ref{case2stated}, \ref{case3stated}, respectively, goes to zero as $N\rightarrow\infty$. Therefore, as $N\rightarrow\infty$, the probability that $p$ and $q$ are both anomalous for $E$ approaches $1$. This completes the proof.
\end{proof}

\section{Conclusion}\label{concl}
Fix an elliptic curve $E$. Then Proposition~\ref{anomProdKorselt} and Proposition 11 from \cite{Korselt} give the following implications.
\begin{center}
\small{product of anomalous primes} $\xrightarrow {\text{Prop.} \; \ref {anomProdKorselt}}$ \small{elliptic Korselt Type I}$\xrightarrow {\text{Prop. } 11}$ {\small elliptic Carmichael}
\end{center}
%Replaced the commented-out block with the block right above this comment -Eric.
{\flushleft These implications, together with Theorem $1.2$ from \cite{Q}, imply the following result.} 

\begin{corollary}
Assuming the Hardy-Littlewood Conjecture, there are infinitely many elliptic Korselt numbers of Type I for the curve $E: y^2 = x^3 + D$, where $D\in \mathbb{Z}$ is neither a square nor a cube in $\mathbb{Q}(\sqrt{-3})$ and $D \neq 80d^6$ for any $d \in \mathbb{Z}[(1+\sqrt{-3})/2]$.
\end{corollary}

In section \ref{type 1}, we explore the strictness of the left-most inclusion in the above diagram. Theorem \ref{gen_silv} establishes deterministic conditions under which an elliptic Korselt number of Type I can be a product of anomalous primes. Furthermore, for $n = pq$ a product of two distinct primes, Theorem \ref{prob_thm} states that nearly all the elliptic curves that make $n$ an elliptic Korselt number of Type I also have $n$ a product of anomalous primes. If an extension to Theorem \ref{prob_thm} can be made to arbitrarily many primes, then in many ways anomalous primes form the building blocks of squarefree elliptic Korselt numbers of Type I.

{\flushleft {\bf Acknowledgments.} The authors would like to thank Lawrence Washington for many helpful discussions and suggestions regarding this work. We also thank Steven J.\ Miller for helpful discussions on an earlier draft of this paper}

\section{\textbf{Appendix.} Numerical Evidence for Proposition \ref{odc}}

We wrote a program that takes as input an integer $N$, randomly choses two distinct primes $5 \le p, q \le N$, and a random elliptic curve whose order over $\F_p$ and $\F_q$ divides $pq + 1 - a_p a_q$. We ran the program $1000$ times, each time recording whether $(p + 1 - a_p)(q + 1 - a_q) = pq + 1 - a_p a_q$. Below, $\Pr(N)$ represents the fraction of the $1000$ times that this property held. 
\vspace{0.1in}

\begin{table}[h]
  \begin{minipage}[h]{1 cm}
\begin{tabular}{cc}
        $N$ & $\Pr(N)$\\
        \hline
        64&.614\\
        128&.725\\
        256&.823\\
        512&.869\\
        1024&.924\\
        2048&.939\\
        4096&.965\\
        8192&.976\\
        16384&.994\\
        32768&.995\\
        65536&.995\\
\end{tabular}

\end{minipage}
 \hspace{3cm}
 \begin{minipage}[h]{10cm}

        \begin{tikzpicture}[scale=0.8]
            \small
            \draw (5, 4) -- (5, 4.4);
            \draw (5, 4.6) -- (5, 10);
            \draw (5, 4) -- (5.4, 4);
            \draw[->] (5.6, 4) -- (16.4, 4);
            \draw (5.4, 3.8) -- (5.4, 4.2);
            \draw (5.6, 3.8) -- (5.6, 4.2);
            \draw (4.8, 4.4) -- (5.2, 4.4);
            \draw (4.8, 4.6) -- (5.2, 4.6);
            \draw[blue] (6,6.1) -- (7,7.3) -- (8, 8.2) -- (9, 8.7) -- (10, 9.2) -- (11, 9.4) -- (12, 9.7) -- (13, 9.8) -- (14, 9.9) -- (15, 10) -- (16, 10);

            \foreach \x/\y in {6/.61, 7/.73, 8/.82, 9/.87, 10/.92, 11/.94, 12/.97, 13/.98, 14/.99, 15/1.0, 16/1.0}{

                \draw[fill=white,draw=blue] (\x,10*\y) circle (2pt) node[above,xshift=-2pt] at (\x,10*\y) {\y};
                \node[anchor=north] at (\x, 3.9) {\x};
                \draw (\x, 3.9) -- (\x, 4);
            }
            \node[anchor=north] at (5, 3.9) {0};
            \draw (5, 3.9) -- (5, 4);

            \foreach \y in {.5, .6, .7, .8, .9, 1}{
                \node[anchor = east] at (4.9, 10*\y) {\y};
                \draw (4.9, 10*\y) -- (5, 10*\y);
            }
            \node[anchor = east] at (4.9, 4) {0};
            \draw (4.9, 4) -- (5, 4);

            \node[anchor=east,text width=1cm,align=left] at (4.3,7) {$\Pr(N)$};
            \node[anchor=north,text width=1cm,align=center] at (10.5,3.4) {$\log_2(N)$};
        \end{tikzpicture}
    \hphantom{$\mathrm{Pr}(N)$}
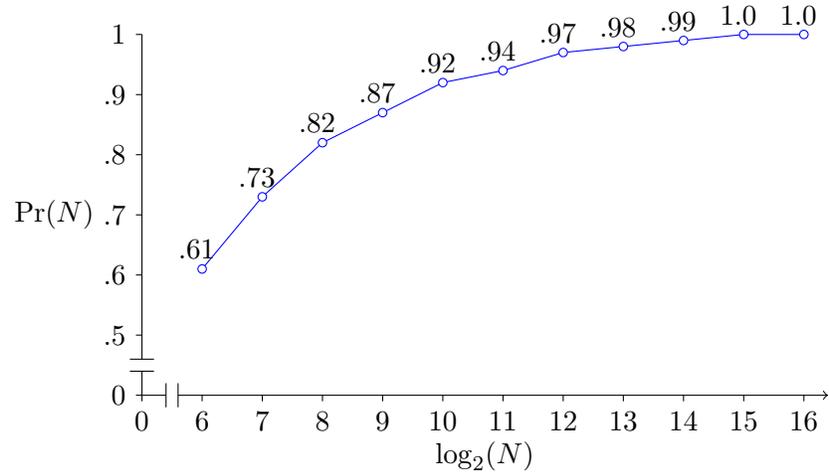
    \caption{\small $\Pr(N)$ v.s. $\log_2(N)$, rounded to two significant digits.} 
 \end{minipage}
\end{table}
\end{document}